\documentclass[12pt]{amsart}
\usepackage{amsfonts}
\usepackage{amssymb, amscd, amsthm}
\usepackage{latexsym}
\usepackage{multirow}
\usepackage{amsmath}
\usepackage[all]{xy}
\usepackage{mathrsfs}
\usepackage{mathtools}
\usepackage{amsbsy}
\usepackage{amsfonts}
\usepackage{mathrsfs}
\usepackage{verbatim}
\usepackage{version}
\usepackage{color}

\newtheorem{theorem}{Theorem}[section]
\newtheorem{lemma}[theorem]{Lemma}
\newtheorem{thm}[theorem]{Theorem}
\newtheorem{prop}[theorem]{Proposition}
\newtheorem{rem}[theorem]{Remark}
\newtheorem{coro}[theorem]{Corollary}

\newcommand{\ra}{\rightarrow}
\newcommand{\mo}{\mathcal{O}}
\newcommand{\mf}{\mathcal{F}}

\newcommand{\mb}{\mathcal{B}}

\newcommand{\mj}{\mathcal{J}}

\newcommand{\mc}{\mathcal{C}}

\newcommand{\Ext}{\operatorname{Ext}}
\newcommand{\Pic}{\operatorname{Pic}}

\def\<{\langle}
\def\>{\rangle}

\newcommand{\ls}{|L|}
\newcommand{\p}{\mathbb{P}}
\newcommand{\bz}{\mathbb{Z}}

\newcommand{\bc}{\mathbb{C}}

\begin{document}
\fontsize{12pt}{14pt} \textwidth=14cm \textheight=21 cm
\numberwithin{equation}{section}
\title{Higher direct images of the structure sheaf via the Hilbert-Chow morphism.}
\author{Yao Yuan}

\begin{abstract}Let $X$ be a projective smooth surface over $\mathbb{C}$ with $H^2(\mathcal{O}_X)=0$.  Let $M=M(L,\chi)$ be the moduli space of 1-dimensional semistable sheaves with determinant $\mathcal{O}_X(L)$ and Euler characteristic $\chi$.  We have the Hilbert-Chow morphism
$\pi:M\rightarrow |L|$.  We give explicit forms of the higher direct images $R^i\pi_*\mathcal{O}_M$ under some mild conditions on $M$ and $|L|$.  Our result shows that $R^i\pi_*\mathcal{O}_M$ are direct sums of line bundles.  In particular, using our result one gets explicit formulas for the Euler characteristic of $\pi^*\mathcal{O}_{|L|}(m)$, which in $X=\mathbb{P}^2$ case was once conjectured by Chung-Moon.

~~~

\end{abstract}

\maketitle
\tableofcontents
\section{Introduction.}
\subsection{The main result.}
We work over the complex number $\bc$.

Let $X$ be a projective smooth surface with $K_X$ the canonical line bundle.  Let $L$ be an effective divisor class on $X$.  Let $M(L,\chi)$ be the moduli space of 1-dimensional semistable sheaves with determinant $\mo_X(L)$ and Euler characteristic $\chi$.  We have the Hilbert-Chow morphism
\[\pi:M(L,\chi)\ra \ls,~~\mf\mapsto\text{supp}(\mf)\]
sending each sheaf $\mf$ to its Fitting support.  We write $M=M(L,\chi)$ for simplicity if there is no confusion on $L$ and $\chi$.

The main purpose of the paper is to study the higher direct image $R^i\pi_*\mo_{M}$ on $\ls$. 
Denote by $M^s$ the subset of $M$ consisting of stable sheaves.  
Denote by $\ls^{int}$ ($\ls^{sm}$, resp.) the subset of $\ls$ consisting of integral curves (smooth curves, resp.).  Let $g_L$ be the arithmetic genus of curves in $\ls$ and $\ell:=\dim\ls$.  
Our main theorem is as follows.
\begin{thm}\label{inthm2}Let $X$ be a surface with $h^1(\mo_X)=q$ and $h^2(\mo_X)=0$.  Let $L$ be an effective divisor class with $L.K_X<0$ and $h^1(\mo_X(-L))=0$.  Assume $M$ has rational singularities 
with $M\setminus \pi^{-1}(\ls^{int})$ of codimension $\geq 2$ 
and assume moreover $\ls\setminus\ls^{int}$ has codimension $\geq 2$ and $\ls^{sm}\neq\emptyset$.  Then for all $i$ we have
\begin{equation}\label{meq00}R^i\pi_*\mo_{M}\cong \wedge^i\left(\mo_{\ls}^{\oplus q}\oplus \mo_{\ls}(-1)^{\oplus g_L-q}\right).\end{equation}\end{thm}
We write the following theorem as a special case for Theorem \ref{inthm2} with $X$ rational.
\begin{thm}\label{inthm1}Let $X$ be a rational surface with $L.K_X<0$.  Assume $M$ has rational singularities 
with $M\setminus \pi^{-1}(\ls^{int})$ of codimension $\geq 2$ 
and assume moreover $\ls\setminus\ls^{int}$ has codimension $\geq 2$ and $\ls^{sm}\neq\emptyset$.  Then for all $i$ we have
\begin{equation}\label{meq0}R^i\pi_*\mo_{M}\cong \mo_{\ls}(-i)^{\oplus\binom{g_L}{i}}.\end{equation}\end{thm}

\begin{coro}\label{inco1}Let $X,M,L,\pi$ be as in Theorem \ref{inthm1}, then we have $H^i(M,\mo_M)=0$ for all $i>0$ and
\[\chi(M,\pi^*\mo_{\ls}(m))=\binom{m+\ell-g_L}{m}=\binom{m-L.K_X-1}{m}.\]
\end{coro}
\begin{proof}By Theorem \ref{inthm1} we have $H^i(R^j\pi_*\mo_M)=0$ except for $i=j=0$ and 
\begin{eqnarray}\chi(M,\pi^*\mo_{\ls}(m))&=&\sum_{i=0}^{g_L}(-1)^i\chi(R^i\pi_*\mo_M(m)) \nonumber\\
&=&\sum_{i=0}^{g_L}(-1)^i\binom{g_L}{i}\binom{m-i+\ell}{\ell}=P_{\ell-g_L,\ell-g_L}(m,g_L),\end{eqnarray}
where $P_{k,l}(m,n)$ is defined in (\ref{defP}).  By Lemma \ref{ebc} we have 
\[\chi(M,\pi^*\mo_{\ls}(m))=\binom{m+\ell-g_L}{\ell-g_L}=\binom{m+\ell-g_L}{m}.\]
The last thing left to show is $\ell-g_L=-L.K_X-1$.  By assumption $\dim M=\dim \pi^{-1}(\ls^{int})=g_L+\ell$.  On the other hand since $L.K_X<0$, $\pi^{-1}(\ls^{int})$ is smooth of the expected dimension $L^2+1$.   We also have $g_L=\chi(\mo_X(-L))-\chi(\mo_X)+1=\frac{L^2+L.K_X}2+1$.  Therefore $\ell=\frac{L^2-L.K_X}2$ and we are done. 
\end{proof}
\begin{rem}\label{inrem1}If $X$ is a del Pezzo surface, then $M$ always has rational singularities since it is a quotient of a smooth scheme.  Moreover $\pi$ is flat over $M^s$ by Corollary 1.3 in \cite{Yuan9}.  Therefore $\text{\emph{codim}}(M^s\setminus\pi^{-1}(\ls^{int}))= \text{\emph{codim}}(\ls\setminus\ls^{int})$. 
\end{rem}
\begin{rem}Corollary \ref{inco1} applies to $X=\p^2,L=dH$ and $(\chi,d)=1$ and proves the conjecture by Chung-Moon (\cite[Conjecture 8.2]{ChM} or \cite[Conjecture 6.13]{KMMP}).
\end{rem}
\begin{rem}Let $X$ be a K3 surface.  If $M=M^s$, then it is smooth and $\pi$ is a Lagrangian fibration.  By Matsushita's result (\cite[Theorem 1.3]{Mat}) we have
\begin{equation}\label{keq0}R^i\pi_{*}\mo_M\cong \Omega^i_{\ls},\end{equation}
where $\Omega^i_{\ls}=\wedge^i\Omega^1_{\ls}$ with $\Omega^1_{\ls}$ the sheaf of differentials over $\ls$.  We will give an explanation in Remark \ref{exdif} on why there is a difference between (\ref{meq00}) and (\ref{keq0}).  
\end{rem}

We will prove our main result in the next section.  

\subsection{Acknowledgements.}
I would like to thank Qizheng Yin for introducing to me the work of Arinkin on compactified Jacobians.  I thank the referees.    


\section{Proof of the main result.}
\subsection{The relative compactified Jacobians.}
In this subsection we prove a property on the higher direct images of the structure sheaf of relative compactified Jacobians to the base.

Let $f:\mc\ra\mb$ be a flat proper family of integral curves with planar singularities of genus $g\geq 1$.  Let $\mc^{sm}\subset \mc$ be the open subscheme where $f$ is smooth.  Let $h^d:\overline{\mj}^d\ra \mb$ be the relative compactified Jacobian (of degree $d$) and let $\mj^d\subset \overline{\mj}^d$ be the relative Jacobian, i.e. for each $b\in\mb$, $\mj^d_b$ parametrizes line bundles with degree $d$ on $\mc_b$.  We assume that $\mc,\mb$ and $\overline{\mj}^d$ are smooth.  
\begin{prop}\label{hdrcj}With notations as above, we have
\begin{equation}\label{cj1}h_{*}^d\mo_{\overline{\mj}^d}\cong \mo_{\mb};\end{equation}
\begin{equation}\label{cj2}R^1h^{d}_*\mo_{\overline{\mj}^d}\cong R^1f_*\mo_{\mc};
\end{equation}
and 
\begin{equation}\label{cj3}R^ih^{d}_*\mo_{\overline{\mj}^d}\cong \wedge^i R^1h^{d}_*\mo_{\overline{\mj}^d},\text{ for any }i=1,\cdots,g.
\end{equation}
\end{prop}   
\begin{proof}By Altman-Iarrobino-Kleiman's result in \cite{AIK}, every fiber of $h^d$ is projective and irreducible and hence $h^{d}_*\mo_{\overline{\mj}^d}\cong \mo_{\mb}$.

To prove (\ref{cj2}), let us first assume that there is a section $s:\mb\ra\mc^{sm}$.  Hence $\overline{\mj}^d$ are naturally isomorphic to each other.  We may let $d=0$ and denote $\overline{\mj}=\overline{\mj}^0,h=h^0$.  By \cite[Theorem B]{Ari2} we see that $\overline{\mj}\cong \overline{\Pic}^0(\overline{\mj})$.  Therefore we have a natural embedding $\imath:\mc\hookrightarrow \overline{\Pic}^0(\overline{\mj})\cong \overline{\mj}$ as schemes over $\mb$.  The surjective morphism $\mo_{\overline{\mj}}\ra\mo_{\mc}$ induces the following morphism
\[R^1h_{*}\mo_{\overline{\mj}}\ra R^1f_*\mo_{\mc},\]
which is an isomorphism restricted to all fibers by \cite[Theorem 1.2 (ii), Proposition 6.1]{Ari1}.  Hence $R^1h_{*}\mo_{\overline{\mj}}\cong R^1f_*\mo_{\mc}$.

If there is no section $s:\mb\ra\mc^{sm}$, we may do an \'etale surjective base change $\sigma:\mb'\ra \mb$ such that $f':\mc^{sm'}\ra\mb'$ has a section.  Then 
$$R^1h^{d'}_*\mo_{\overline{\mj}^{d'}}\cong \sigma^*R^1 h^{d}_*\mo_{\overline{\mj}^d}\cong R^1f'_*\mo_{\mc'}\cong \sigma^*R^1f_*\mo_{\mc}.$$
Therefore $R^1 h^{d}_*\mo_{\overline{\mj}^d}\cong R^1f_*\mo_{\mc}$ because $\sigma$ is \'etale and surjective.

(\ref{cj3}) is just a relative version of  Arinkin's result (\cite[Theorem 1.2 (ii)]{Ari1}).  Firstly by \cite[Theorem 1.2 (ii)]{Ari1} we have both $R^ih^{d}_*\mo_{\overline{\mj}^d}$ and $\wedge^i R^1h^{d}_*\mo_{\overline{\mj}^d}$ are locally free of the same rank,  hence it is enough to show (\ref{cj3}) holds in codimension 1.  Therefore we may assume $\mb$ to be a curve.  

Let $D\subset \mb$ consist of finite reduced points such that $\mc_x$ is smooth for all $x\not\in D$.  Take a resolution $\eta:\widetilde{\mj}^d\ra \overline{\mj}^d$ such that $\eta$ is an isomorphism restricted to $(h^d)^{-1}(\mb\setminus D)$ and $(\eta\circ h^d)^{-1}{D}$ is a normal crossing divisor on $\widetilde{\mj}$.  Since $(h^d)^{-1}(D)$ is reduced, we have that $(\eta\circ h^d)^{-1}{D}$ is reduced.  We also have $R^i\eta_*\mo_{\widetilde{\mj}^d}=0$ for $i>0$ and $\eta_*\mo_{\widetilde{\mj}^d}\cong \mo_{\overline{\mj}^d}$ because $\overline{\mj}^d$ is smooth.  Therefore it is enough to show (\ref{cj3}) with $h^d$ replaced by $\eta\circ h^d$.  (\ref{cj3}) holds over $\mb\setminus D$ by the functoriality of variations of Hodge structures.  Since $(\eta\circ h^d)^{-1}{D}$ is a reduced normal crossing divisor, the monodromy of the variations of the Hodge structures are unipotent. Therefore there are functorial canonical extensions of Hodge bundles (see \cite[Theorem 6.16]{Sch} or \cite[Corollary 11.18]{PeSt}) and hence (\ref{cj3}) extends to the whole $\mb$  (see also \cite[Lemma 4.9]{Mat}).   

\end{proof}
\subsection{Moduli spaces of 1-dimensional sheaves on a surface.}
Now let $M=M(L,\chi)$ parametrize 1-dimensional semistable sheaves with determinant $\mo_X(L)$ and Euler characteristic $\chi$ on a projective surface $X$.  Let $\pi:M\ra\ls$ be the Hilbert-Chow morphism.   Let $g_L$ be the arithmetic genus of curves in $\ls$. 
We have the following proposition on the dualizing sheaf $\omega_M$ of $M$.
\begin{prop}\label{cbm}Let $L.K_X<0$ and $\ls^{sm}\neq \emptyset$, then the canonical line bundle over $\pi^{-1}(\ls^{int})$ is $\pi^{*}\mo_{\ls}(1)^{\otimes L.K_X}|_{\pi^{-1}(\ls^{int})}$.  In particular if $M$ is normal and Cohen-Macaulay and $M\setminus \pi^{-1}(\ls^{int})$ has codimension $\geq 2$, then the dualizing sheaf $\omega_M$ of $M$ is $\pi^{*}\mo_{\ls}(1)^{\otimes L.K_X}$.
\end{prop} 
\begin{proof}The proof of the first statement is analogous to \cite[Proposition C.0.15]{Yuan1}.  Notice that $\Ext^2(\mf,\mf)=0$ for any $\mf\in\pi^{-1}(\ls^{int})$ by $L.K_X<0$, and hence $\pi^{-1}(\ls^{int})$ is smooth. 

Now if $M$ is normal and Cohen-Macaulay, then $\omega_M$ is the push forward of the canonical line bundle of its smooth locus.  Hence the proposition.
\end{proof}
\begin{prop}\label{htf}Let $L.K_X<0$ and $\ls^{sm}\neq\emptyset$.  If $M$ has rational singularities and $M\setminus \pi^{-1}(\ls^{int})$ has codimension $\geq 2$, then $R^i\pi_*\mo_M=0$ except for $i=0,1,\cdots,g_L$, and $R^i\pi_*\mo_M$ is locally free for all $i=0,1,\cdots,g_L$.
\end{prop}
\begin{proof}Since $M$ has rational singularities, we can take a resolution $\xi:\widetilde{M}\ra M$ such that $\xi$ is birational, $\widetilde{M}$ is smooth and $R^i\xi_*\omega_{\widetilde{M}}=R^i\xi_*\mo_{\widetilde{M}}=0$ for all $i>0$ and $\xi_*\mo_{\widetilde{M}}\cong \mo_M$, $\xi_*\omega_{\widetilde{M}}\cong\omega_M$ where $\omega_{\widetilde{M}}$ is the canonical line bundle on $\widetilde{M}$.  Therefore we may replace $M$ by $\widetilde{M}$ and assume that $M$ is smooth.

By Koll\'ar's result (\cite[Theorem 2.1 (i) (ii)]{Ko1}), we have $R^i\pi_*\omega_M=0$ except for $i=0,1,\cdots,g_L$, and $R^i\pi_*\omega_M$ is torsion free for all $i=0,1,\cdots,g_L$.   On the other hand, $\omega_M\cong \pi^{*}\mo_{\ls}(1)^{\otimes L.K_X}$ by Proposition \ref{cbm}.  Hence $R^i\pi_*\mo_M$ are torsion-free.  Then by \cite[Proposition 3.12]{Ko2}, $R^i\pi_*\mo_M$ are Cohen-Macaulay sheaves on $\ls$ hence locally free since $\ls$ is smooth.
\end{proof}

\subsection{Proof of Theorem \ref{inthm2}.}In this subsection, we let $X$ be a surface with $h^1(\mo_X)=q$ and $h^2(\mo_X)=0$.  Let $L$ be an effective divisor class with $h^1(\mo_X(-L))=0$.  Denote by $\mc_L\subset X\times \ls$ the universal curve.  
\begin{equation}\label{embnr}\xymatrix{
  \mathcal{C}_L  \ar@{^{(}->}[r]
                & X\times \ls \ar[ld]^{q} \ar[d]^{p}  \\
                X &\ls             }.\end{equation}  
\begin{lemma}\label{lr1nr}$R^1p_*\mo_{\mc_L}\cong \mo_{\ls}^{\oplus q}\oplus\mo_{\ls}(-1)^{\oplus g_L-q}$.
\end{lemma}
\begin{proof}We have an exact sequence on $X\times\ls$
\begin{equation}\label{rucnr}0\rightarrow q^{*}\mathcal{O}_{X}(-L)\otimes p^{*}\mathcal{O}_{\ls}(-1)\rightarrow\mathcal{O}_{X\times\ls}\rightarrow\mathcal{O}_{\mathcal{C}}\rightarrow0.\end{equation}
Push (\ref{rucnr}) forward to $\ls$ and we have
\[0\ra R^1p_*\mo_{X\times \ls}\ra R^1p_*\mo_{\mc_L}\ra R^2p_*(q^*\mo_X(-L)\otimes p^*\mo_{\ls}(-1))\ra 0.\]
Hence
\[0\ra \mo_{\ls}^{\oplus q}\ra R^1p_*\mo_{\mc_L}\ra \mo_{\ls}(-1)^{\oplus g_L-q}\ra 0.\]
and
$$R^1p_*\mo_{\mc_L}\cong \mo_{\ls}^{\oplus q}\oplus\mo_{\ls}(-1)^{\oplus g_L-q}.$$
The lemma is proved.
\end{proof}
\begin{proof}[Proof of Theorem \ref{inthm2}]Since $\pi^{-1}(\ls^{int})$ and $\ls^{int}$ are smooth, $\pi$ restricted to $\pi^{-1}(\ls^{int})$ is flat.  Combine Proposition \ref{hdrcj} and Lemma \ref{lr1nr} and we have over $\ls^{int}$
\begin{equation}\label{hdint}R^i\pi_*\mo_{\pi^{-1}(\ls^{int})}\cong R^i\pi_*\mo_M|_{\ls^{int}}\cong \wedge^i\left(\mo_{\ls}^{\oplus q}\oplus\mo_{\ls}(-1)^{\oplus g_L-q}\right)\big|_{\ls^{int}}.
\end{equation}

Hence the theorem since $R^i\pi_*\mo_M$ are locally free and $\ls\setminus\ls^{int}$ has codimension $\geq2$.

\end{proof}
\begin{rem}\label{exdif}If $X$ is a K3 surface with $L$ an effective divisor class.  Denote by $\mc_L\subset X\times \ls$ the universal curve.  If $\ls^{int}\neq\emptyset$, then $\ls\cong \p^{g_L}$ and we have the following exact sequence on $\ls$
\[0\ra R^1p_*\mo_{\mc_L}\ra R^2p_*(q^*\mo_X(-L)\otimes p^*\mo_{\ls}(-1))\ra R^2p_*\mo_{X\times \ls}\ra0.\]
Therefore
\[0\ra R^1p_*\mo_{\mc_L}\ra \mo_{\ls}(-1)^{g_L+1}\ra \mo_{\ls}\ra0,\]
which is compatible with the fact that $R^1\pi_*\mo_M\cong \Omega^1_{\ls}$ for $M$ smooth.  So the main reason to cause the difference between (\ref{meq00}) and (\ref{keq0}) is whether the cohomology group $H^2(\mo_X)$ vanishes.
\end{rem}
 
\subsection{An equality in binomial coefficients.}
In this subsection we prove an equality in binomial coefficients used in the proof of Corollary \ref{inco1}.  Lemma \ref{ebc} below maybe obvious to experts, but we still give a proof here.  For $k,l,m,n\in\bz_{\geq0}$ define
\begin{equation}\label{defP}P_{k,l}(m,n):=\sum_{i=0}^n(-1)^i\binom{n}{i}\binom{m+n-i+k}{n+l}.\end{equation}
\begin{lemma}\label{ebc}For any $k,l,m,n\in\bz_{\geq0}$ we have
\[P_{k,l}(m,n)=\binom{m+k}{l}.\]
\end{lemma}
\begin{proof}We prove the lemma in several steps.

\emph{Step 1:} We show that $P_{0,0}(m,n)=1$ for all $m,n\in\bz_{\geq 0}$.

It is easy to see that $P_{0,0}(m,n)=P_{0,0}(n,m)$ and $P_{0,0}(0,n)=P_{0,0}(n,0)=1$ for all $n\in\bz_{\geq0}$.  Now we assume $P_{0,0}(m,n)=1$ if either $n<n_0$ or $m<m_0$, it is enough to show $P_{0,0}(m_0,n_0)=1$.  Use the equality $\binom{m+1}{k}=\binom{m}{k}+\binom{m}{k-1}$ and we get
\[P_{0,0}(m_0,n_0)=P_{0,0}(m_0-1,n_0)-P_{0,0}(m_0-1,n_0-1)+P_{0,0}(m_0,n_0-1).\]
Hence by induction assumption we get $P_{0,0}(m_0,n_0)=1$.  We are done.

\emph{Step 2:} We show that $P_{0,l}(m,n)=\binom{m}{l}$ for all $m,n\in\bz_{\geq 0}$.

It is easy to see $P_{0,l}(m,n)=0$ for $m<l$ and $P_{0,l}(l,n)=1=\binom{l}{l}$.  We assume $P_{0,l}(m,n)=\binom{m}{l}$ if either $l<l_0$ or $m<m_0$, it is enough to show $P_{0,l_0}(m_0,n)=\binom{m_0}{l_0}$.  We have
\[P_{0,l_0}(m_0,n)=P_{0,l_0-1}(m_0-1,n-1)+P_{0,l_0}(m_0-1,n).\]
Hence $P_{0,l_0}(m_0,n)=\binom{m_0-1}{l_0-1}+\binom{m_0-1}{l_0}=\binom{m_0}{l_0}$.  We are done.

\emph{Step 3 (the last step):} We show that $P_{k,l}(m,n)=\binom{m+k}{l}$ for all $m,n\in\bz_{\geq 0}$.

We do induction on $k$.  Assume $P_{k,l}(m,n)=\binom{m+k}{l}$ for $k<k_0$, it is enough to show $P_{k_0,l}(m,n)=\binom{m+k_0}{l}$.  We have
\[P_{k_0,l}(m,n)=P_{k_0-1,l}(m,n)+P_{k_0-1,l-1}(m,n).\]
Hence $P_{k_0,l}(m,n)=\binom{m+k_0-1}{l}+\binom{m+k_0-1}{l-1}=\binom{m+k_0}{l}$.  The lemma is proved.
\end{proof}

Yao Yuan\\
Beijing National Center for Applied Mathematics,\\
Academy for Multidisciplinary Studies, \\
Capital Normal University, 100048, Beijing, China\\
E-mail: 6891@cnu.edu.cn.
\end{document}